\documentclass[a4paper]{amsart}
\usepackage{graphicx}
\usepackage{etex}
\usepackage{amsthm,stmaryrd}
\usepackage{amssymb}
\usepackage{array}
\usepackage{tabu}
\usepackage{epsfig}
\usepackage[usenames,dvipsnames]{color}
\usepackage{verbatim}
\usepackage{hyperref}
\usepackage{mathrsfs}
\usepackage[all]{xy}
\usepackage{xcolor}
\usepackage{float}
\usepackage{enumerate}
\usepackage{tikz}   
\usepackage{changepage} 
\usetikzlibrary{arrows,calc,positioning,decorations.pathreplacing}
\usepackage{relsize}
\usepackage{pst-node}
\usepackage{tikz-cd} 
\usepackage{cite}

\newcommand{\Aut}{\operatorname{Aut}}

\newcommand{\Z}{\mathbb{Z}}

\newcommand{\C}{\mathbb{C}}
\newcommand{\N}{\mathbb{N}}

\newcommand{\FF}{\mathbb{F}}

\newcommand{\F}{\mathcal{F}}
\newcommand{\D}{\mathcal{D}}

\newcommand{\qs}{q}

\newtheorem{definition}{Definition}[subsection]
\newtheorem{theorem}[definition]{Theorem}

\newtheorem*{theorem*}{Theorem}

\newtheorem{proposition}[definition]{Proposition}
\newtheorem{lemma}[definition]{Lemma}
\newtheorem{remark}[definition]{Remark}
\newtheorem{notation}[definition]{Notation}

\newtheorem{corollary}[definition]{Corollary}
\newcounter{exo} \newcounter{numexercice}
\renewcommand{\theexo}{\arabic{exo}}

\newcounter{IntroCounter}
\stepcounter{IntroCounter}

\begin{document}
\title[ADO invariants directly from partial traces of homological representations]{ADO invarints directly from partial traces of homological representations} 
  \author{Cristina Ana-Maria Anghel}
\address{Mathematical Institute, University of Oxford, Oxford, United Kingdom} \email{palmeranghel@maths.ox.ac.uk} 
  \thanks{ }
  \date{\today}

\begin{abstract}
The ADO invariants are a sequence of non-semisimple quantum invariants coming from the representation theory of the quantum group $U_q(sl(2))$ at roots of unity. Ito showed that these invariants are sums of traces of quotients of homological representations of braid groups (truncated Lawrence representations). In this paper we show a direct homological formula for the ADO invariants, as sums of partial traces of Lawrence type representations, without further truncations. 
\end{abstract}

\maketitle
\setcounter{tocdepth}{1}
 \tableofcontents
 {
\section{Introduction} 

This paper concerns the family of non-semisimple quantum invariants called Coloured Alexander polynomials or ADO invariants \cite{ADO}. They come from the representation theory of the quantum group $U_q(sl(2))$ at roots of unity. The theory of quantum invariants started with the discovery of the Jones polynomial and was developed further by Reshetikhin and Turaev. They introduced an algebraic construction which starts with the representation theory of a quantum group and leads to link invariants. This construction applied for the generic quantum group $U_q(sl(2))$ leads to the sequence of coloured Jones polynomials. Dually, the representation theory of the same quantum group at roots of unity leads to the sequence of coloured Alexander invariants $\{ \Phi_N(L,\lambda)\}_{N \in \mathbb N}$. This family recovers the original Alexander polynomial at the first term (corresponding to $N=2$).

The initial definition of these invariants was algebraic. On the homological side, Lawrence introduced in \cite{Law}, \cite{Law1} a sequence of braid group representations using the homology of coverings of configuration spaces in the punctured disc. This construction opened a new direction towards connections between quantum invariants and topological information. Using these homological representations, Lawrence (\cite{Law1}) and later Bigelow based on her work (\cite{Big}), showed that the Jones polynomial can be seen as an intersection pairing between homology classes in certain coverings of configuration spaces. We refer to such a description as a {\em topological model}. 

A different approach aims to describe an invariant as sums of {\em traces of homological representations}.
In 2015 (\cite{Ito3}), Ito showed that the loop expansion of coloured Jones invariants can be obtained as an infinite sum of traces of Lawrence representations. A second step towards this kind of homological information was made in \cite{Ito2}, where Ito introduced certain quotients of the Lawrence representations. Then, he proved that the ADO invariants are sums of {\em traces of truncated Lawrence representations}. He said (\cite{Ito2}) that it would be interesting to investigate further links towards more topological definitions, mentioning that an obstruction to doing this in his model comes from a ``lack of understanding'' of the truncation procedure of the Lawrence representations for $N>2$. 

Further on, in 2017 (\cite{Cr2}) we constructed a {\em topological model} for the coloured Jones polynomials, as intersections of homology classes in coverings of configuration spaces. Then, in 2019 (\cite{Cr1}), we constructed a {\em topological model} for the coloured Alexander polynomials using truncated Lawrence representations. These were existence type results. In \cite{Martel}, Martel introduced a slight variation of Lawrence representations. Based on that, in his thesis he gave a homological model for the coloured Jones polynomials, as sums of {\em traces of these homological representations}. 

Going back to {\em topological models}, very recently we showed a globalised result (\cite{Cr3}), proving that the coloured Jones and ADO invariants both come as different specialisations of a {\em unified topological model} over two variables. More precisely, we described them as specialisations of certain intersection pairings between explicit homology classes in these Lawrence type representations. 

This paper concerns a {\em homological type model} for the ADO invariants, and so it is different and independent from the {\em topological intersection type models} which appeared in \cite {Cr1} and \cite{Cr3}. This model will be given by a sum of partial traces of certain homological representations. More specifically, we introduce the notion of {\em homological partial trace}, which corresponds to the usual partial trace, but restricted to weight spaces rather than tensor powers of quantum representations. Then, we  provide a {\em homological model} for the ADO invariants, as sums of homological partial traces of specialisations of certain subrepresentations in Lawrence representations. 


\subsection{Strategy and description of the homology groups}

For $n,m \in \N$, we use the unordered configuration space of $m$ points in the $n$-punctured disc and denote it by $C_{n,m}$. Then $\tilde{C}_{n,m}$ is a certain $\Z \oplus \Z$-covering space of this configuration space. Further on, $H^{\text{lf},-}_m(\tilde{C}_{n,m}, \Z)$ will the Borel-Moore homology of the covering, relative to a certain part of the boundary. This homology group is generated by homology classes given by lifts of geometric submanifolds $\mathbb F_e$ in the base space $C_{n,m}$, prescribed by partitions $e$ of $m$ into $n$ positive integers. The precise construction is presented in section \ref{S:3}. The version of Lawrence representation from \cite{Martel} is obtained from the braid group action on this homology:
$$l_{n,m}: B_n\rightarrow Aut\left(H^{\text{lf},-}_m(\tilde{C}_{n,m}, \Z), \Z[x^{\pm1}, d^{\pm 1}]\right).$$
Now, we fix $N\in \mathbb N$ to be the colour of the ADO invariant that we want to study and $\xi_N$ the standard primitive $2N^{th}$ root of unity.  

Further on, we define a subspace in this homology, generated by those classes given by lifts of the geometric submanifolds $\mathbb F_e$ where $e$ is an $n-$partition of $m$, whose components are all strictly smaller than $N$ (see definition \ref{D:3}):
$${}^N \! H_{n,m} \subseteq_{\Z[x^{\pm1}, d^{\pm 1}]} H^{\text{lf},-}_m(\tilde{C}_{n,m}, \Z).$$
 
 \begin{notation}(Specialisation)\label{D:1} For two complex numbers $q,\lambda \in \C$, we use the specialisation of coefficients towards one variable, given by the formula:
$$\psi_{q,\lambda}: \Z[x^{\pm},d^{\pm}]\rightarrow \C$$ 
$$\psi_{q,\lambda}(x)= \qs^{2 \lambda}; \ \ \psi_{q,\lambda}(d)=\qs^{-2}.$$
\end{notation}
Using the precise form of the $R$-matrix from the quantum side, we show that  $l_{n,m}$ induces a well defined braid group action on the subspace ${}^N \! H_{n,m}$, when we specialise through $\psi_{\xi_N,\lambda}$ (definition \ref{D:2}). We call this action {\em level $N$ Lawrence representation} and denote it by:
$$  {}^N \! l_{n,m}:B_n\rightarrow Aut\left( {}^N \! H_{n,m}|_{\psi_{\xi_N,\lambda}}, \C \right).$$

Further on, we define the notion of homological partial trace of weight zero, and denote it as below (definition \ref{D:4}):
$$ hptr_0: Aut\left( {}^N \!  H_{n,m}|_{\psi_{\xi_N,\lambda}}, \C \right) \rightarrow \C$$

Our main result presents the $N^{th}$ ADO invariant as a sum of partial traces of the level $N$ Lawrence representations. 

\begin{theorem}($N^{th}$ ADO invariant from level N Lawrence representations) \label{THEOREM}\\
Let us fix $N \in \N$ be the colour of the invariant and $\lambda \in \C$. Then if an oriented knot $L$ is a closure of a braid $\beta_n \in B_n$, we have the following homological formula:
\begin{equation}
\begin{aligned}
\Phi_{N}(L,\lambda) ={\xi_N}^{(N-1)\lambda w(\beta_n)} \cdot & \xi_{N}^{(n-1)(1-N)\lambda} \cdot \\
& \cdot \mathlarger{\sum}_{m=0}^{(N-1)(n-1)} \xi_{N}^{-2m(1-N)}  
\ hptr_0\left( {}^{N} \! l_{n,m} (\beta_n) \right)
\end{aligned}
\end{equation}
In this formula, $w(\beta_n)$ is the writhe of the braid.
\end{theorem}
One feature of this formula is that it partially answers and explains Ito's question. It does not have any truncation and so the coefficients of the ADO invariants are more clear from the homological point of view. 

Also, we would like to remark that this model is somehow a counterpart of Martel's result for the coloured Jones polynomials, both given as {\em sums of traces on homological representations}. This one concerns the root of unity case while Martel's model concerns the generic case.  
\subsection{Comparison and idea of the proof}
The difference between this model and the model from \cite{Ito2} is that the former is a sum of traces on quotients of homological representations. Here, we do not take any quotients but instead we pay the price of having a partial homological trace. However, the direct homological actions have an advantage from the geometrical  point of view and also the model provides a very concrete algorithm for computations. The explanation ``behind the scenes'' of this difference is that Ito's strategy is based on highest weight spaces and ours is based on weight spaces. Then the result from \cite{Ito2} uses a Kohno's type identification for quantum representations on highest weight spaces and homological representations. We use a more explicit Kohno's type identification due to Martel \cite{Martel} for quantum representations on weight spaces and homological representations.

More specifically, in \cite{Ito2} the splitting of tensor powers of the quantum representation at roots of unity is used by means of highest weight spaces. Then the homological correspondent for these highest weight spaces at roots of unity is constructed, which has to have a truncation. On the other hand, we use weight spaces instead of highest weight spaces, and the fact that their direct sum recovers the same tensor power of the quantum representation at roots of unity. A technical point is that we set up the whole construction such that it comes from the quantum group over two variables. Then, we show that the weight spaces at roots of unity, naturally correspond to subrepresentations in homological representations rather than quotients. The last technical part is to define the concept of partial trace in a homological context. We do this by translating the effect of the quantum trace on tensor powers of the quantum representation at roots of unity, restricted onto certain weight spaces, which in turn correspond to homological representations.  
\subsection{Structure of the paper}

 In Section \ref{S:2}, we describe the version of the quantum group that we use (which has divided powers of one generator), as well as the representation theory correlated to it. Then, we show that we can see the ADO invariants from a certain specialisation of the Verma module associated to the two-variable quantum group. In the last part of the section, we show explicitely the formula that we will use on the quantum side. Then, in Section \ref{S:3}, we set up the homological part of the picture, and define the Lawrence representation that we work with. Section \ref{S:4} is devoted to the construction of a subrepresentation of the Lawrence representation, which we call {\em level $N$ Lawrence representation}. After that, we define the notion of {\em homological partial trace of weight zero} on this Lawrence subrepresentation. In the last part, in Section \ref{S:5}, we put these notions together and prove a homological model for the $N^{th}$ ADO invariant as a weighted sum of homological partial traces on the level $N$ Lawrence representations.  
 
\subsection*{Acknowledgements} 
 
 This paper was prepared at the University of Oxford, and I acknowledge the support of the European Research Council (ERC) under the European Union's Horizon 2020 research and innovation programme (grant agreement No 674978).

\section{Notations}  
In the paper, we will use certain specialisations of coefficients of some homology modules. The precise definition of a specialisation is the following.
\begin{notation}(Specialisation)\label{N:spec}\\ 
Let $R$ be a ring and $M$ an $R$-module. Suppose that we fix a basis $\mathscr B$ of the module $M$.
If $S$ is another ring, let us assume that we fix a specialisation of the coefficients, given by a ring morphism:
$$\psi: R \rightarrow S.$$
Then, the specialisation of the module $M$ by the change of coefficients $\psi$ is the following $S$-module: $$M|_{\psi}:=M \otimes_{R} S$$ which has the corresponding basis given by $$\mathscr B_{M|_{\psi}}:=\mathscr B \otimes_{R}1 \in M|_{\psi}. $$
\end{notation}
\begin{notation}(Partial traces)
For a tensor power of two vector spaces $V,W$, we denote the partial trace with respect to the elements of $W$ by:
$$ptr_V:End(V \otimes W)\rightarrow End(V).$$
\end{notation}
Also, for the representation theory part, we will use the following notations for quantum factorials:
$$ \{ x \} :=\qs^x-\qs^{-x} \ \ \ \ [x]_{\qs}:= \frac{\qs^x-\qs^{-x}}{\qs-\qs^{-1}}$$
$$[n]_\qs!=[1]_\qs[2]_\qs...[n]_\qs$$
$${n \brack j}_\qs=\frac{[n]_\qs!}{[n-j]_\qs! [j]_\qs!}.$$
\section{The ADO invariant from generic Verma modules}\label{S:2}
In this part, we present the set-up on the representation theory side which we use, in order to obtain the ADO invariant. The usual definition of these invariants starts from the version of the quantum group $U_q(sl(2))$ over one variable, which is a root of unity. Here, we will use the version of the quantum group over two variables, given by the divided powers of one of the  generators.  We showed in \cite{Cr3} that we can specialise this representation theory later on, towards one variable, and obtain the ADO invariant in this manner. For the precise construction of this set-up, we refer to \cite{Cr3}, Section 3. Below, we outline the main points of the construction. 

\begin{definition}(Quantum group)
Let $U_q(sl(2))$ be the Hopf algebra over the ring with two variables $\Z[q^{\pm 1},s^{\pm 1}]$, genererated by $\{ E,F^{(n)}, K^{\pm1} | \ n \in \mathbb N^{*} \}$ subject to the relations:
\begin{equation*}
\begin{cases}
KK^{-1}=K^{-1}K=1; \ \ \ KE=\qs^2EK; & \ \ \ KF^{(n)}=\qs^{-2n} F^{(n)}K;\\
F^{(n)}F^{(m)}= {n+m \brack n}_{\qs} F^{(n+m)}&\\
[E,F^{(n+1)}]=F^{(n)}(q^{-n}K-q^{n}K^{-1}).
\end{cases}
\end{equation*} 
\end{definition}
This quantum group has an associated Verma module, given by the following description:
\begin{definition}(\cite{Cr3}, \cite{JK}) (Generic Verma module)

Let $\hat{V}$ be an infinite dimensional $\Z[q^{\pm 1}, s^{\pm 1}]$-module generated by a sequence of vectors $\{v_0, v_1,...\}$, with the following $U_q(sl(2))$-actions:
\begin{equation}\label{action}  
\begin{cases}
Kv_i=s\qs^{-2i}v_i,\\
Ev_i=v_{i-1},\\
F^{(n)}v_i = {n+i \brack i}_\qs \prod _{k=0}^{n-1} (s\qs^{-k-i}-s^{-1} \qs^{k+i}) v_{i+n}.
\end{cases}
\end{equation}
\end{definition}
Further on, in order to construct knot invariants, we need to use finite dimensional representations. We start with a natural number $N \in \N$ and $\xi_N$ the standard primitive $2N^{th}$ root of unity. 
\begin{notation}
Let us fix $\lambda \in \C$. For the pair
$(q=\xi_N=e^{\frac{2 \pi i}{2N}}, \lambda \in \C )$ we consider the specialisation
\begin{equation}
\begin{cases}
 \eta_{\xi_N,\lambda}:\Z[q^{\pm},s^{\pm}]\rightarrow\C \\
\eta_{\xi_N,\lambda}(s)= {\xi}^{\lambda }_N.
\end{cases}
\end{equation}
\end{notation}
For the non-semisimple invariants at roots of unity, we will start with the set-up over two variables, and then specialise through $\eta_{\xi_N,\lambda}$.

\begin{definition}
We denote the quantum group at roots of unity by $$U_{\xi_N}(sl(2))=U_q(sl(2))\otimes_{\eta_{\xi_N,\lambda}}\C.$$
Correspondingly, we consider the specialised Verma module $$\hat{V}_{\xi_N,\lambda}=\hat{V}\otimes_{\eta_{\xi_N,\lambda}}\C$$
which is a representation of $U_{\xi_N}(sl(2))$. 
Moreover, let us consider the vector subspace generated by the first $N$ vectors over $\C$:
$$U_{\lambda}:=<v_{0},...,v_{N-1}>_{\C} \subseteq \hat{V}_{\xi_N,\lambda}.$$
\end{definition}
The quantum group $U_q(sl(2))$ has an $R-$matrix over $\Z[q^{\pm 1}, s^{\pm 1}]$, which gives a braiding $\mathscr R$, leading to a braid group representation (as presented in \cite{Cr3} Section 3.1):
\begin{equation}
\begin{aligned}
\hat{\varphi}_n: B_n \rightarrow Aut_{U_q(sl(2))}\left(\hat{V}^{\otimes n}\right) \ \ \ \ \ \ \ \ \ \ \ \ \ \ \ \ \ \ \ \ \ \ \\
\ \ \ \ \ \ \ \ \ \ \ \ \ \sigma _i^{\pm 1}\rightarrow Id_V^{\otimes (i-1)}\otimes \mathscr R^{\pm 1}  \otimes Id_V^{\otimes (n-i-1)}.
\end{aligned}
\end{equation}
Further on, this action will induce an action at roots of unity provided by the formula:
$$\ \ \hat{\varphi}^{\xi_N,\lambda}_{n}  :B_n \rightarrow \Aut( \hat{V}_{\xi_N,\lambda}^{\otimes n}). $$
A key point in \cite{Cr3} is based on the remark that even if $U_{\lambda}$ is not a sub-representation of $\hat{V}_{\xi_N,\lambda}$ over the quantum group $U_{\xi_N}(sl(2))$, the braid group action $\hat{\varphi}^{\xi_N,\lambda}_{n}$ behaves well with respect to this inclusion. More precisely, we have the following property. 
\begin{proposition} (\cite{Cr3}-Lemma 3.1.7)( Braid group action at roots of unity)\label{P:0}
  
 The action $\hat{\varphi}^{\xi_N,\lambda}_{n}$ on $\hat{V}_{\xi_N,\lambda}^{\otimes n}$ preserves the vector subspace $U_{\lambda}^{\otimes n}$, meaning that following diagram commutes:
 
\begin{center}
\begin{tikzpicture}
[x=1.2mm,y=1.4mm]

\node (b1)  [color=blue]             at (0,10)    {$U_{\lambda}^{\otimes n}$};
\node (t1) [color=black] at (30,10)   {$\hat{V}_{\xi_N,\lambda}^{ \otimes n}$};
\node (b2) [color=blue] at (0,0)  {$\hat{\varphi}^{\xi_N,\lambda}_{n}$};
\node (t2)  [color=black]             at (30,0)    {$\hat{\varphi}^{\xi_N,\lambda}_{n}$};
\node (d2) [color=black] at (15,10)   {$\hookrightarrow$};
\node (d2) [color=black] at (15,12)   {$\iota$};
\node (d2) [color=black] at (0,5)   {$\circlearrowleft$};
\node (d2) [color=black] at (30,5)   {$\circlearrowleft$};
\node (d2) [color=black] at (15,5)   {$\equiv$};
\end{tikzpicture}
\end{center} 
 \end{proposition}
 \begin{notation}
 We denote the restriction of the braid group action $\hat{\varphi}^{\xi_N,\lambda}_{n}$ onto the finite dimensional part by:
 $$\varphi^{\xi_N,\lambda}_{n}  :B_n \rightarrow \Aut_{\C}(  {U_{\lambda}}^{\otimes n}).$$
 \end{notation}
 \subsection{Weight spaces}
In this part, we consider certain subspaces inside the tensor powers of quantum representations, called weight spaces, which will have nice homological correspondents.  
 \begin{definition}(Weight spaces)  
 Let us fix $n,m\in \N$ and a natural number $N \in \N$ (which will correspond to the colour of the invariant). We consider the following.
 
{\bf 1) Generic weight spaces }\\
The $n^{th}$ weight space of weight $m$ corresponding to the generic Verma module $\hat{V}$:
\begin{equation}
\hat{V}_{n,m}:= \{v \in \hat V^{\otimes n} \mid Kv=s^nq^{-2m}v \}.
\end{equation}
The $n^{th}$ weight space of weight $m$ corresponding to the $N^{th}$ finite part inside $\hat{V}$:
\begin{equation}
V^N_{n,m}:=\left( \hat{V}_{n,m} \ \cap <v_0,...,v_{N-1}>_{\Z[q^{\pm1},s^{\pm 1}]} \right) \subseteq \hat V^{\otimes n}.
\end{equation}

{\bf 2) Weight spaces at roots of unity }\\
We introduce the notion of the $n^{th}$ weight space of weight $m$ corresponding to the finite dimensional part at roots of unity, which is given by:
\begin{equation}
 V^{\xi_N,\lambda}_{n,m}:=\{v\in U^{\otimes n}_{\lambda} \mid Kv=\xi_N^{n\lambda-2m}v \}\subseteq U^{\otimes n}_{\lambda}.
 \end{equation}
\end{definition}
\begin{notation}
Let us consider the following indexing sets:
$$E_{n,m}=\{e=(e_1,...,e_{n})\in \N^{n} \mid e_1+...+e_{n}=m \}.$$
$$E^N_{n,m}=\{e=(e_1,...,e_{n})\in E_{n,m} \mid 0 \leq e_1,...,e_n \leq N-1 \}.$$
\end{notation}
\begin{remark}(Basis for weight spaces)\\
A basis for the generic weight space is given by:
$$\mathscr{B}_{\hat{V}_{n,m}}=\{v_e:= v_{e_1}\otimes... \otimes v_{e_n} \mid  e=(e_1,...,e_n) \in E_{n,m}\}.
$$
A basis for the $N^{th}$ finite weight space is given by:
$$\mathscr{B}_{V^N_{n,m}}=\{v_e:= v_{e_1}\otimes... \otimes v_{e_n} \mid e=(e_1,...,e_n) \in E^N_{n,m}\}.
$$
\end{remark}
\begin{remark}(Braid group actions \cite{Cr3})
The generic action $\hat{\varphi}_n$ induces a braid group representation onto generic weight spaces, which we denote by: $$\hat{\varphi}_{n,m}:B_n \rightarrow \Aut( \hat{V}_{n,m}).$$
\end{remark}
 We notice that there is a problem if we look carefully at the definition of the $R$-matrix over two variables, because the corresponding action does not preserve the inclusion of the $N^{th}$ finite weight space over two parameters into the generic one:
$$V^N_{n,m}\subseteq \hat{V}_{n,m}.$$ 

However, something interesting happens if we specialise $q$ to a $2N^{th}$ root of unity. This comes from the specific formula of the specialisation of the $R-$matrix, towards one variable. We refer to \cite{Cr3} (Section 3.2) for the details of this phenomenon. The main property is the following.
\begin{proposition}(\cite{Cr3})\label{P:1} 
The specialised action at roots of unity given by $$\hat{\varphi}_{n,m}|_{\eta_{\xi_N,\lambda}}: B_n\rightarrow \Aut( \hat{V}_{n,m}|_{\eta_{\xi_N},\lambda})$$ preservers the inclusion of the specialised $N^{th}$ weight space into the specialisation of the generic one, leading to a braid group representation, which we denote by:
\begin{equation}
\begin{aligned}
& \varphi^N_{n,m}|_{\eta_{\xi_N,\lambda}}: B_n\rightarrow \Aut\left( V^N_{n,m}|_{\eta_{\xi_N,\lambda}}\right)\\
& \varphi^N_{n,m}|_{\eta_{\xi_N,\lambda}}=\hat{\varphi}_{n,m}|_{\eta_{\xi_N,\lambda}} \big( \text{ restricted to }V^N_{n,m}|_{\eta_{\xi_N,\lambda}} \big).
\end{aligned}
\end{equation}
\end{proposition} 
In the following we will see that one advantage of these level $N$ weight spaces is that they are defined over two variables. Secondly, their direct sum recovers the whole tensor power of the representation $U_{\lambda}$, as below.
\begin{proposition}
The specialisations of the $N^{th}$ weight spaces through $\eta_{\xi_N,\lambda}$ recover the tensor power of the module at roots at unity as below:
\begin{equation}
U^{\otimes n}_{\lambda}= \bigoplus _{m=0}^{n(N-1)} V_{n,m}^N|_{\eta_{\xi_N,\lambda}}.
 \end{equation}
\end{proposition}
\begin{proof}
We start with the remark that: 
\begin{equation}\label{eq:1}
\hat{V}^{\otimes n}=\bigoplus _{m=0}^{\infty} \hat{V}_{n,m}.
\end{equation}
Further on, this direct sum recovers the tensor power of the vector space generated by the first $N$-vectors:
\begin{equation}\label{eq:2}
\left( <v_0,...,v_{N-1}>_{\Z[s^{\pm1},q^{\pm 1}]}\right)^{\otimes n}=\bigoplus _{m=0}^{n(N-1)} V^{N}_{n,m}.
\end{equation}
This comes from equation \eqref{eq:1} and the remark that the maximal weight that we can rich with the first $N$ vectors is $n(N-1)$. Then, by specialising  the coefficients from equation \eqref{eq:2} through $\eta_{\xi_N,\lambda}$, we conclude the decomposition from the statement. 
\end{proof}
Now, we will use that the braid group actions presented above preserve weight spaces (for details, see \cite{Cr3}). This shows that the splitting from equation \eqref{eq:1} is preserved by the braid group action. This means that we have the following property.
\begin{corollary}\label{C:3} 
The braid group actions onto the specialisations at roots of unity of the $N^{th}$ weight spaces recover the whole braid group action at roots at unity as in the following formula:
\begin{equation}
\varphi^{\xi_N,\lambda}_{n}= \bigoplus _{m=0}^{n(N-1)} \varphi^N_{n,m}|_{\eta_{\xi_N,\lambda}}.
 \end{equation}
\end{corollary}

\subsection{Construction of the ADO invariant}
In this part, we will see that we can construct the $N^{th}$ ADO invariant from the module $U_{\lambda}$. This set-up appeared in \cite{Cr3} and is a new way of describing this invariant, which comes from the quantum group over two variables and uses the trick of considering $U_{\lambda}$ in the way that we have described above (is a module rather than a representation over the quantum group).
\begin{notation}
For $ f \in End(U_{\lambda})$, which is a scalar times identity, we denote this scalar as follows:
$$ f=<f> \cdot Id \in \C \cdot Id_{U_{\lambda}}.$$
\end{notation}
\begin{proposition}(The $N^{th}$ ADO invariant)

Let us fix $N \in \N$ to be the colour and $\lambda \in \C$. We denote $\xi_N=e^{\frac{2\pi i}{2N}}$. Then, if an oriented knot $L$ is a closure of a braid $\beta_n \in B_n$, we have:
\begin{equation}\label{eq:A}
\begin{aligned}
\Phi_{N}(L,\lambda)=&{\xi_N}^{(N-1)\lambda w(\beta_n)}  <ptr_{U_{\lambda}}\left( (Id \otimes K^{1-N})\circ \varphi^{\xi_N,\lambda}_{n}(\beta_n)\right)>.
\end{aligned}
\end{equation}
\begin{proof}
First of all, we have to explain why the endomorphism from this formula is a scalar times the identity.
\begin{equation}
ptr_{U_{\lambda}}\left( (Id \otimes K^{1-N})\circ \varphi^{\xi_N,\lambda}_{n}(\beta_n)\right) \in \C \cdot Id_{U_{\lambda}} \subseteq End(U_{\lambda}).
\end{equation}

For the usual version of the quantum group $U_q(sl(2))$ at roots of unity, this is not at all surprising. It comes from the fact that this partial trace (which is exactly the Reshetikhin-Turaev construction applied on the knot $L$ with one strand that is cut) gives an automorphism of the represnetation $U_{\lambda}$ {\em over the quantum group}. Further on, it is known from representation theory properties that $U_{\lambda}$ is a simple representation of the quantum group. This means that any endomorphism of this representation is a scalar times identity.

However, in our case, $U_{\lambda}$ is not a representation over the quantum group $U_{\xi_N}(sl(2))$, it is just a vector space over $\C$. We remind the construction of our braid group action, from the generic one presented in Proposition \ref{P:0}: 

$$\varphi^{\xi_N,\lambda}_{n}=\hat{\varphi}_{n}|_{\eta_{\xi_N,\lambda}} \text{ restricted to } U_{\lambda}^{\otimes n}.$$ The key point is that the formulas for the action $\hat{\varphi}_{n}|_{\eta_{\xi_N,\lambda}}$ and the braid group action coming from the usual construction, presented in \cite{Ito2}, are actually the same (they are just used on different bases in our case and the case from \cite{Ito2}). Based on this correspondence and the property mentioned above, telling us that the usual Reshetikhin-Turaev corresponding to this partial quantum trace leads to a scalar, we conclude that also the partial trace from above leads to a scalar.

Further on, the formula looks similar to the formula from Section 2.3 presented in \cite{Ito2}, but the difference occurs from fact in that paper, the construction of the representation at roots of unity $U_{\lambda}$ is different than ours. However, based on the above discussion about the precise actions of $R$-matrices at the level of braid group representations, the action from \cite{Ito2} and the one that we use are actually the same (when they are computed on the standard basis).

\end{proof}

\end{proposition}
\section{Homological representations}\label{S:3}
In this section, we present the version of the homological representations of braid groups that we will use for our model. We will follow \cite{Martel}.
Let us fix $n,m \in \N$. Then, we denote by $\D^2\subseteq \C$ the unit disc including its boundary.

For our construction, we will use the $n-$punctured disc:
$$D_n:=\D^2 \setminus \{1,...,n \} $$
We denote by $C_{n,m}$ the unordered configuration space of $m$ points in the $n$-punctured disc, given by:
$$C_{n,m}=Conf_m(D_n)= \Big( D^{\times m}_n \setminus \{x=(x_1,...,x_m) \mid \ \exists \ i,j \ \text { such that } x_i=x_j\}\Big) / Sym_m$$
(here, $Sym_m$ is the symmetric group of order $m$).\\
Let us fix $m$ points on the boundary of the disc $d_1,...,d_m \in  \partial D_n$ and let us denote by ${\bf d}:=\{ d_1,...,d_m \} \in C_{n,m}$ to be the corresponding point in the configuration space.
\vspace{-3mm}
\begin{figure}[H]
\begin{tikzpicture}
[scale=4.3/5]
\foreach \x/\y in {0/2,2/2,4/2,2/1,2.6/1,3.6/1.15} {\node at (\x,\y) [circle,fill,inner sep=1pt] {};}
\node at (0.2,2) [anchor=north east] {$1$};
\node at (2.2,2) [anchor=north east] {$i$};
\node at (4.2,2) [anchor=north east] {$n$};
\node at (3,2) [anchor=north east] {$\sigma_i$};
\node at (2.2,1) [anchor=north east] {$d_1$};
\node at (2.9,1.02) [anchor=north east] {$d_2$};
\node at (4,1.05) [anchor=north east] {$d_m$};
\node at (2.68,2.3) [anchor=north east] {$\wedge$};
\draw (2,1.8) ellipse (0.4cm and 0.8cm);
\draw (2,2) ellipse (3cm and 1cm);
\foreach \x/\y in {7/2,9/2,11/2,8.5/1,9.5/1,10.5/1.10} {\node at (\x,\y) [circle,fill,inner sep=1pt] {};}
\node at (7.2,2) [anchor=north east] {$1$};
\node at (9.2,2) [anchor=north east] {$i$};
\node at (11.2,2) [anchor=north east] {$n$};
\node at (9,1) [anchor=north east] {$d_1$};
\node at (9.7,1.02) [anchor=north east] {$d_2$};
\node at (10.9,1.05) [anchor=north east] {$d_m$};
\node at (8.5,1.7) [anchor=north east] {$\delta$};
\draw (9,2) ellipse (3cm and 1cm);
\draw (9.5,1)  arc[radius = 5mm, start angle= 0, end angle= 180];
\draw [->](9.5,1)  arc[radius = 5mm, start angle= 0, end angle= 90];
\draw (8.5,1) to[out=50,in=120] (9.5,1);
\draw [->](8.5,1) to[out=50,in=160] (9.16,1.19);
\end{tikzpicture}
\caption{}
\label{fig2}
\end{figure}
\vspace{-7mm}
\subsection{Covering space}
In the sequel, we will use certain loops in the configuration space. We denote by $\sigma_i\in \pi_1(C_{n,m})$ the class represented by the loop in $C_{n,m}$ with $m-1$ fixed components (the base points $d_2$,...,$d_m$) and the first one going on a loop in $D_n$ around the $i^{th}$ puncture. Then, $\delta \in \pi_1(C_{n,m})$ will be the class of the loop in the configuration space with the last $(m-2)$ components constant and the first two components which swap the points $d_1$ and $d_2$, as in figure \ref{fig2}.  
\begin{notation}
Let us consider $aug:\Z^n\rightarrow \Z$ to be the map given by:
$$aug(x_1,...,x_m)=x_1+...+x_m.$$
\end{notation}
\begin{definition}(Local system)\\
Let $\rho: \pi_1(C_{n,m}) \rightarrow H_1\left( C_{n,m}\right)$ be the abelianisation map. For $m \geq 2$,  the homology of the configuration space is (see \cite{Ito2}): 
\begin{equation*}
\begin{aligned}
H_1\left( C_{n,m}\right) \simeq & \ \ \Z^{n} \ \oplus \ \Z \ \ \\
& \langle \rho(\sigma_i) \rangle \  \langle \rho(\delta) \rangle,  \ \ \ \ \ {i\in \{1,...,n\}}.
\end{aligned}
\end{equation*}
Combining the two morphisms, we consider the local system:
\begin{equation}
\begin{aligned}
\phi: \pi_1(C_{n,m}) \rightarrow \ & \Z \oplus \Z \ \ \ \ \ \ \ \ \ \\ 
 & \langle x \rangle \ \langle d \rangle\\
\hspace{3mm} \phi= (aug \oplus Id_{\Z})& \circ \rho.  
\end{aligned}
\end{equation}
\end{definition}
\begin{definition}(Covering of the configuration space)
Let $\tilde{C}_{n,m}$ be the covering of $C_{n,m}$ corresponding to the local system $\phi$. Then, the deck transformations of this covering have two variables and they are given by:
$$Deck(\tilde{C}_{n,m},C_{n,m})\simeq <x><d>.$$
\end{definition}
Let us fix ${\bf \tilde{d}} \in \tilde{C}_{n,m}$ be a lift of the base point ${\bf d}=\{d_1,...,d_m\}$ in the covering.
\begin{figure}[H]
\begin{center}
\begin{tikzpicture}\label{pic'}
[x=0.7mm,y=0.03mm,scale=0.1/5,font=\Large]
\foreach \x/\y in {-1.2/2, 0.4/2 , 1.3/2 , 2.5/2 , 3.6/2 } {\node at (\x,\y) [circle,fill,inner sep=1.3pt] {};}
\node at (-1,2) [anchor=north east] {$w$};
\node at (0.6,2.5) [anchor=north east] {$1$};
\node at (0.2,1.55) [anchor=north east] {$\color{red}\eta^e_1$};
\node at (1,2) [anchor=north east] {$\color{red} \eta^e_{e_1}$};
\node at (3.8,1.7) [anchor=north east] {$\color{red} \eta^e_{m}$};
\node (dn) at (8,1.5) [anchor=north east] {\Large \color{red}$\eta_e=(\eta^e_1,...,\eta^e_m)$};
\node at (2,5.1) [anchor=north east] {\Large \color{red}$\tilde{\eta}_e$};

\node at (4,2.5) [anchor=north east] {n};
\node at (0.8,0.8) [anchor=north east] {\bf d$=$};
\node at (0.8,4.4) [anchor=north east] {\bf $\bf \tilde{d}$};
\node at (0.3,2.6) [anchor=north east] {$\color{black!50!red}\text{Conf}e_1$};
\node at (3.6,2.6) [anchor=north east] {$\color{black!50!red}\text{Conf} e_{n}$};
\node at (2.1,3) [anchor=north east] {\huge{\color{black!50!red}$\FF_e$}};
\node at (2.1,6.2) [anchor=north east] {\huge{\color{black!50!red}$\tilde{\FF}_e$}};
\node at (-2.5,2) [anchor=north east] {\large{$C_{n,m}$}};
\node at (-2.5,6) [anchor=north east] {\large{$\tilde{C}_{n,m}$}};
\draw [very thick,black!50!red,-][in=-160,out=-10](-1.2,2) to (0.4,2);
\draw [very thick,black!50!red,->] [in=-145,out=-30](-1.2,2) to (3.6,2);
 \draw[very thick,black!50!red] (2.82, 5.6) circle (0.6);
\draw (2,2) ellipse (3.2cm and 1.3cm);
\draw (2,5.4) ellipse (3cm and 1.11cm);
\node (d1) at (1.3,0.8) [anchor=north east] {$d_1$};
\node (d2) at (1.9,0.8) [anchor=north east] {$d_{e_1}$};
\node (dn) at (2.8,0.8) [anchor=north east] {$d_m$};
\draw [very thick,dashed, red,->][in=-60,out=-190](1.2,0.7) to  (-0.2,1.9);
\draw [very thick,dashed,red,->][in=-70,out=-200](1.3,0.7) to (0,1.9);
\draw [very thick,dashed,red,->][in=-90,out=0](2.5,0.7) to (3,1.6);
\draw [very thick,dashed,red,->][in=-70,out=-200](0.8,4.4) to (3,5);
\end{tikzpicture}
\end{center}
\caption{}
\label{fig3}
\end{figure}
The construction uses the Borel-Moore homology of this covering space. For the sequel,  let us fix a point $w \in \partial  D_n$. 
\begin{definition}\label{R:1}  (\cite{Martel})
Let $H^{\text{lf},-}_m(\tilde{C}_{n,m}, \Z)$  be the Borel-Moore homology relative to part of the boundary which is represented by the fiber in $\tilde{C}_{n,m}$ over the base point $w$ (more precisely, the points in the configuration space $\tilde{C}_{n,m}$ whose projection onto $C_{n,m}$ contains $w$). This is a module over the group ring of the deck transformations, namely $\Z[x^{\pm 1}, d^{\pm 1}]$.
\end{definition}
From the property that the braid group is the mapping class group of the punctured disc and the precise form of the above local system, there is an induced action on this homology, which is compatible with its module structure:
$$B_n\curvearrowright H^{\text{lf},-}_m(\tilde{C}_{n,m}, \Z)= \Z[x^{\pm 1}, d^{\pm 1}]-\text{module}.$$
\subsection{Lawrence representation}
\begin{definition}({\bf Multiarcs} \cite{Martel})\\
a) Let us start with a partition $e \in E_{n,m}$. For each $i \in \{ 1,...,n \} $, we use a segment in $D_n$ starting at the point $w$ and finishing at the $i^{th}$ puncture, as in figure \ref{fig3}. Then, we consider the space of ordered configurations of $e_i$ points on this segment.
Further on, we denote the projection onto the unordered configuration space by:
$$\pi_m : { D}^{\times m}_n \setminus \{x=(x_1,...,x_m)| x_i=x_j \}) \rightarrow C_{n,m}$$
The product of these ordered configuration spaces on segments together with the  projection lead to a submanifold in the unordered configuration space, denoted by: 
$$\FF_e:=\pi_m (Conf_{e_1} \times ... \times Conf_{{e}_{n}})\subseteq C_{n,m}$$
b) We consider an additional imput, given by a fixed set of paths between the base point on the boundary and the red segments: $${\eta}^{e}_k: [0,1] \rightarrow D_n, k \in \{1,...,m\}$$ as in figure \ref{fig3}.
The set of all paths $\eta^e_k$ leads to a path in the configuration space, denoted by:
$$\eta^e := \pi_m \circ (\eta^{e}_1, ..., \eta^{e}_m ) : [0,1] \rightarrow C_{n,m}.$$
We remark that:
\begin{equation}
\begin{cases}
\eta^e(0)={\bf d} \\
\eta^e(1)\in \FF_e.
\end{cases}
\end{equation}
Further on, let $\tilde{\eta}^e$ be the unique lift of the path $\eta^e$ with the property  that:
\begin{equation}
\begin{cases}
\tilde{\eta}^e:  [0,1] \rightarrow \tilde{C}_{n,m}\\
\tilde{\eta}^e(0)={ \bf \tilde{ d}}.
\end{cases}
\end{equation}
\end{definition}
\begin{definition} (Multiarcs)\\ \label{f}
Let us consider $\tilde{\FF}_{e}$ to be the unique lift of the submanifold $\FF_e$ with the property:
\begin{equation}
\begin{cases}
\tilde{\FF}_{e}: (0,1)^m\rightarrow \tilde{C}_{n,m}\\ 
\tilde{\eta}^e(1) \in \tilde{\FF}_{e}.
\end{cases}
\end{equation}
This submanifold gives a class in the Borel-Moore homology, denoted by:
$$[\tilde{\FF}_{e}] \in H^{\text{lf},-}_m(\tilde{C}_{n,m}, \Z).$$ 
This is called the multiarc corresponding to the partition $e \in E_{n,m}$.
 \end{definition}
\begin{proposition}(\cite{Martel})
The set of all multiarcs 
$\{  [\tilde{\FF}_e] \ | \  e\in E_{n,m} \}$ is a basis for $ H^{\text{lf},-}_m(\tilde{C}_{n,m}, \Z).$
\end{proposition}

\begin{notation}(Normalised multiarc)
For $e \in E_{n,m}$, let us consider a normalisation of the multiarc given by:
$$\F_{e}:=x^{\sum_{i=1}^{n}(i-1) e_i} [\tilde{\FF}_e]  \in H^{\text{lf},-}_m(\tilde{C}_{n,m}, \Z).$$
\end{notation}

\begin{notation}(Lawrence representation)\label{T1}\\
Let $l_{n,m}$ be the braid group action from above, in the basis $\mathscr{B}_{H^{\text{lf},-}_m(\tilde{C}_{n,m}, \Z)}:=\{ \F_{e}, e \in E_{n,m}\}$ given by multiarcs:
$$l_{n,m}: B_n\rightarrow Aut(H^{\text{lf},-}_m(\tilde{C}_{n,m}, \Z)).$$
\end{notation}
\subsection{Identification between weight space representations and homological representations}

We will use the following specialisation of coefficients:
\begin{equation}
\begin{cases}
\gamma: \Z[x^{\pm1},d^{\pm1}]\rightarrow \Z[q^{\pm1},s^{\pm1}]\\
\gamma(x)=s^2; \ \ \gamma(d)=q^{-2}.
\end{cases}
\end{equation}
The advantage of the basis from above is that it naturally corresponds to the basis in the weight spaces. More precisely, in \cite{Martel}, it was shown the following identification.
\begin{theorem}(\cite{Martel})
The quantum representations on weight spaces are isomorphic to the homological representations of the braid group:
$$B_n\curvearrowright \hat{V}_{n,m} \simeq H_{n,m}|{_{\gamma}}\curvearrowleft B_n$$ 
$$\left( \hat{\varphi}_{n,m}, \mathscr{B}_{\hat{V}_{n,m}} \right) \ \ \ \ \ \ \left( l_{n,m}|_{\gamma}, \mathscr{B}_{H^{\text{lf},-}_m(\tilde{C}_{n,m}, \Z)}  \right) $$
\begin{equation}
 \ \ \ \ \ \ \ \ \ \ \ \ \ \ \ \ \Theta_{n,m}(v_{e_1}\otimes ... \otimes v_{e_{n}})=\F_{e},  \text{ for }  e=(e_1,...,e_n) \in E_{n,m}.
\end{equation}
\end{theorem}
\begin{corollary}\label{C:1}
This isomorphism shows that the corresponding specialisations at roots of unity are isomorphic:
\begin{equation}
\hat{\varphi}_{n,m}|_{\eta_{\xi_N,\lambda}} \simeq l_{n,m}|_{\psi_{\xi_N,\lambda}}
\end{equation}
\end{corollary}
\section{Level N Lawrence representation}\label{S:4}
On the quantum side, we have seen that the $N^{th}$ ADO invariant is related to weight spaces corresponding to the $N^{th}$ finite part of the generic Verma module. Having this in mind, we consider a subspace in the Lawrence representation, which will correspond to the level $N$ weight spaces from the quantum side. Let us make this precise. 
\begin{definition}(Level $N$ Lawrence representation)\label{D:3}\\  
Let us consider the subspace in the Lawrence representation generated by the  multiarcs whose multiplicities are all bounded by $N$:
\begin{equation}
{}^N \! H_{n,m}:=<[\FF_{e}]\mid e \in E^N_{n,m}>_{\Z[x^{\pm 1}, d^{\pm 1}]} \subseteq H^{\text{lf},-}_m(\tilde{C}_{n,m}, \Z). 
\end{equation}
\end{definition}
In the sequel, we show that this homological subspace, up to level $N$, corresponds to the weight spaces from tensor powers of the finite dimensional module of dimension $N$ from the Verma module, after an appropriate specialisation. 
\begin{lemma}
The braid group action on $H^{\text{lf},-}_m(\tilde{C}_{n,m}, \Z)$ specialised through $\psi_{\xi_N,\lambda}$ preserves the specialised vector subspace ${}^N \! H_{n,m}|_{\psi_{\xi_N,\lambda}}$:
\begin{center}
\begin{tikzpicture}
[x=1.2mm,y=1.4mm]

\node (b1)  [color=blue]             at (0,10)    {${}^N \! H_{n,m}|_{\psi_{\xi_N,\lambda}}$};
\node (t1) [color=black] at (34,10)   {$H^{\text{lf},-}_m(\tilde{C}_{n,m}, \Z)|_{\psi_{\xi_N,\lambda}}$};
\node (b2) [color=black] at (30,0)  {$l_{n,m}|_{\psi_{\xi_N,\lambda}}$};
\node (t2)  [color=blue]             at (0,0)    {$l_{n,m}|_{\psi_{\xi_N,\lambda}}$};
\node (d2) [color=black] at (15,10)   {$\hookrightarrow$};
\node (d2) [color=black] at (15,12)   {$\iota$};
\node (d2) [color=black] at (0,5)   {$\circlearrowleft$};
\node (d2) [color=black] at (30,5)   {$\circlearrowleft$};
\node (d2) [color=black] at (15,5)   {$\equiv$};
\end{tikzpicture}
\end{center} 

\end{lemma}
\begin{proof}
Going back to the algebraic side, we know that the braid group action  specialised at roots of unity using $\eta_{\xi_N,\lambda}$, preserves the small weight spaces inside the ones corresponding to the Verma module, as presented in Proposition \ref{P:1}:
\begin{center}
\begin{equation}\label{eq:3}
\begin{tikzpicture}
[x=1.2mm,y=1.4mm]

\node (b1)  [color=blue]             at (0,10)    {$V^N_{n,m}|_{\eta_{\xi_N,\lambda}}$};
\node (t1) [color=black] at (30,10)   {$\hat{V}_{n,m}|_{\eta_{\xi_N,\lambda}}$};
\node (b2) [color=black] at (30,0)  {$\hat{\varphi}_{n,m}|_{\eta_{\xi_N,\lambda}}$};
\node (t2)  [color=blue]             at (0,0)    {$\varphi^N_{n,m}|_{\eta_{\xi_N,\lambda}}$};
\node (d2) [color=black] at (15,10)   {$\hookrightarrow$};
\node (d2) [color=black] at (15,12)   {$\iota$};
\node (d2) [color=black] at (0,5)   {$\circlearrowleft$};
\node (d2) [color=black] at (30,5)   {$\circlearrowleft$};
\node (d2) [color=black] at (15,5)   {$\equiv$};
\end{tikzpicture}
\end{equation}
\end{center} 
Here, we remind that $\varphi^N_{n,m}|_{\eta_{\xi_N,\lambda}}$ is $\hat{\varphi}_{n,m}|_{\eta_{\xi_N,\lambda}}$ restricted to the specialised weight space $V^N_{n,m}|_{\eta_{\xi_N,\lambda}}$. 

Now, using the identification from Corollary \ref{C:1} we notice that the  generators given by monomials from $V^N_{n,m}|_{\eta_{\xi_N,\lambda}}$ corresponds exactly to the normalised multiarcs which are prescribed by partitions with all components smaller than $N$.  

This shows that $V^N_{n,m}|_{\eta_{\xi_N,\lambda}}$ corresponds to the homological module ${}^N \! H_{n,m}|_{\psi_{\xi_N,\lambda}}$ which is specialised through ${\psi_{\xi_N,\lambda}}$:
\begin{center}
\begin{equation}\label{eq:4}
\begin{tikzpicture}
[x=1.2mm,y=1.4mm]
\node (b1)  [color=blue]             at (0,10)    {$\hat{V}_{n,m}|_{\eta_{\xi_N,\lambda}}$};
\node (t1) [color=black] at (30,10)   {$H^{\text{lf},-}_m(\tilde{C}_{n,m}, \Z)|_{\psi_{\xi_N,\lambda}}$};
\node (b2) [color=black] at (30,0)  {${}^N \! H_{n,m}|_{\psi_{\xi_N,\lambda}}$};
\node (t2)  [color=blue]             at (0,0)    {$V^N_{n,m}|_{\eta_{\xi_N,\lambda}}$};
\node (d2) [color=black] at (13,10)   {$\simeq$};
\node (d2) [color=black] at (15,16)   {$\Theta_{n,m} \ \mid _{\eta_{\xi_N},\lambda}$};
\node (d2) [color=black] at (0,5)   {$\cup$};
\node (d2) [color=black] at (30,5)   {$\cup$};
\node (d2) [color=black] at (13,0)   {$\leftrightarrow$};
\end{tikzpicture}
\end{equation}
\end{center} 
Using the commutativity property from equation \eqref{eq:3} and the correspondence presented in relation \eqref{eq:4}, we conclude the commutativity property for the braid group actions on the homological side.
\end{proof}
\begin{proposition}(Level $N$ Lawrence representation) \label{D:2}\\ 
It follows that $l_{n,m}|_{\psi_{\xi_N,\lambda}}$ induces a well defined action on the specialised subspace ${}^N \!  H_{n,m}|_{\psi_{\xi_N,\lambda}}$, which we denote by:
$${}^N \! l_{n,m}: B_n\rightarrow Aut\left({}^N \!  H_{n,m}|_{\psi_{\xi_N,\lambda}}\right).$$
\end{proposition}
\begin{corollary}\label{C:2}
We have the following identifications of braid group actions:
\begin{equation}
\varphi^N_{n,m}|_{\eta_{\xi_N,\lambda}}\curvearrowright  V^N_{n,m}|_{\eta_{\xi_N,\lambda}} \simeq {}^N \!  H_{n,m}|_{\psi_{\xi_N,\lambda}} \curvearrowleft {}^N \! l_{n,m}.
\end{equation}
\end{corollary}
\subsection{Homological Partial Trace}
In this part, we introduce the concept of {\em partial trace on Lawrence representations}. This definition comes from the aim of having a correspondent of the partial trace from the quantum side, in homological terms. In the end, we will see that the collection of all homological partial traces will correspond to the partial trace on the quantum side (on the tensor power of the representation $U_{\lambda}$).

\begin{notation}
Let ${}^N \! H^{0}_{n,m}$ be the subspace in the level $N$ Lawrence representation which is generated by those multiarcs which are prescribed by partitions whose first component is zero:
 $${}^N \! H^{0}_{n,m}:=<[\FF_{e}]\mid e \in E^N_{n,m}, e_1=0>_{\Z[x^{\pm 1}, d^{\pm 1}]} \subseteq {}^N \! H_{n,m}. 
$$
Let $\iota_0: {}^N \! H^{0}_{n,m}|_{\psi_{\xi_N,\lambda}}\rightarrow {}^N \! H_{n,m} |_{\psi_{\xi_N,\lambda}}$ be the corresponding inclusion.
Further on, we consider the projection onto this subspace, as follows:
$$\pi_0: {}^N \! H_{n,m}|_{\psi_{\xi_N,\lambda}}\rightarrow {}^N \! H^{0}_{n,m}|_{\psi_{\xi_N,\lambda}}$$
$$\pi_0(\FF_e)=\begin{cases}[\FF_e], \text{ if } e_1=0\\
0, \text{ otherwise}.
\end{cases}
$$
\end{notation}
 Now, we introduce the definition of a homological partial trace (with respect to elements from the basis of ${}^N \! H_{n,m}$ prescribed by partitions starting with zero) on the level $N$ Lawrence representation.  
\begin{definition}(Homological partial trace)\\ \label{D:4}  
Let $n,m\in \N$. The weight zero partial trace (corresponding to the last $m-1$ components) is given by:
\begin{equation}
\begin{aligned}
& hptr_0: Aut\left( {}^N \!  H_{n,m}|_{\psi_{\xi_N,\lambda}} \right) \rightarrow \C\\
& hptr_0(f)=tr\left(\pi_0\circ f \circ \iota_0 \right).
\end{aligned}
\end{equation}

\end{definition}
\section{Homological model for the ADO invariant}\label{S:5}
In this section, we show the homological model presented in Theorem \ref{THEOREM}.
We start with the definition of the ADO polynomial presented in formula \eqref{eq:A}:
\begin{equation}
\begin{aligned}
\Phi_{N}(L,\lambda)&={\xi_N}^{(N-1)\lambda w(\beta_n)}  <ptr_{U_{\lambda}}\left( (Id \otimes K^{1-N})\circ \varphi^{\xi_N,\lambda}_{n}(\beta_n)\right)>.
\end{aligned}
\end{equation}
In the following, we investigate the partial trace from this formula. We remind that the corresponding endomorphism is a scalar times the identity:
\begin{equation}\label{eq:5} 
ptr_{U_{\lambda}}\left( (Id \otimes K^{1-N})\circ \varphi^{\xi_N,\lambda}_{n}(\beta_n)\right) \in \C \cdot Id_{U_{\lambda}} \subseteq End(U_{\lambda}).
\end{equation}

Following this remark, we have that:
\begin{equation}
\begin{aligned}
ptr_{U_{\lambda}}\left( (Id \otimes K^{1-N})\circ \varphi^{\xi_N,\lambda}_{n}(\beta_n)\right)(v_i)& = \\ 
 =~<ptr_{U_{\lambda}}\left( (Id \otimes K^{1-N}) \circ \varphi^{\xi_N,\lambda}_{n}(\beta_n)\right)> v_i, & \ \forall i \in \{0,...,N-1\}.
\end{aligned}
\end{equation}
Now, going back to the formula for the $N^{th}$ ADO invariant, we will compute the partial trace on the base vector $v_0$. In order to do this, we  denote by $pr_0$ the projection onto the subspace generated by the vector $v_0$:
$$pr_0:U_{\lambda}\rightarrow \C$$
\begin{equation}
pr_0(v_i)=
\begin{cases} 1, \text{ if } i=0\\
0, \text{otherwise}.
\end{cases}
\end{equation}

Using this, we have the formula:
\begin{equation}
\begin{aligned}
\Phi_{N}(L,\lambda)& ={\xi_N}^{(N-1)\lambda w(\beta_n)} \ pr_0 \circ ptr_{U_{\lambda}}\left( (Id \otimes K^{1-N}) \circ \varphi^{\xi_N,\lambda}_{n}(\beta_n)\right)(v_0).
\end{aligned}
\end{equation}

In the next part, for a fixed basis $B$ and a vector $v$, we denote by $C(v,w)$ the coefficient of $v$ in the expression obtained for the vector $w$ written in the basis $B$. We apply this notation for the basis given by the standard tensor monomials in the tensor product of $U_{\lambda}$. Using this notation, the partial trace with respect to $v_0$ can be written as follows:
\begin{equation}
 \Phi_{N}(L,\lambda)={\xi_N}^{(N-1)\lambda w(\beta_n)} \ \cdot
\end{equation}
\begin{equation*} 
\cdot \sum_{i_2,...,i_{n-1}=0}^{N-1} \hspace{-5mm} C\left( v_{0}\otimes v_{i_2} \otimes ... \otimes v_{i_{n}},  \left( Id \otimes K^{1-N}\right) \circ \varphi^{\xi_N,\lambda}_{n}(\beta_n)\left(v_{0}\otimes v_{i_2} \otimes ... \otimes v_{i_{n}}\right) \right).
\end{equation*}

Having in mind the notion of weight spaces, we split the above sum corresponding to the total weight $ m \in \{0,..,N-1\}$ of the vectors, as below:
\begin{equation}\label{eq:6}
 \Phi_{N}(L,\lambda)={\xi_N}^{(N-1)\lambda w(\beta_n)} \ \cdot \mathlarger{\sum}_{m=0}^{(N-1)(n-1)} \\ 
\end{equation}
\begin{equation*}
 \left( \hspace{-2mm} \sum_{\substack {(i_2,...,i_{n-1}) \\   \ \ \ \ \ \ \in E^N_{n-1,m}}} C\left( v_{0}\otimes v_{i_2} \otimes ... \otimes v_{i_{n}},  \left( Id \otimes K^{1-N}\right) \circ \varphi^{\xi_N,\lambda}_{n}(\beta_n)(v_{0}\otimes v_{i_2} \otimes ... \otimes v_{i_{n}} \right) \hspace{-1mm} \right).
\end{equation*}

Now, we remind that the braid group action on the quantum side behaves well with respect to weight spaces, as discussed in Corollary \ref{C:3}:
\begin{equation}
\varphi^{\xi_N,\lambda}_{n}= \bigoplus _{m=0}^{n(N-1)} \varphi^N_{n,m}|_{\eta_{\xi_N,\lambda}}.
 \end{equation}
 
This splitting through actions on weight spaces together with equation \eqref{eq:6}, lead to the following description:
\begin{equation}
 \Phi_{N}(L,\lambda)={\xi_N}^{(N-1)\lambda w(\beta_n)} \  \cdot \mathlarger{\sum}_{m=0}^{(N-1)(n-1)} 
\end{equation} 
\begin{equation*}
 \left(\hspace{-3mm} \sum_{{\substack {(i_2,...,i_{n-1}) \\   \ \ \ \ \ \ \in E^N_{n-1,m}}}} \hspace{-4mm} C\left( v_{0}\otimes v_{i_2} \otimes ... \otimes v_{i_{n}},  \left( Id \otimes K^{1-N}\right) \circ \varphi^N_{n,m}|_{\eta_{\xi_N,\lambda}}(\beta_n)(v_{0}\otimes v_{i_2} \otimes ... \otimes v_{i_{n}} \right)\hspace{-1mm} \right).
\end{equation*}

Now, we look what happens for a fixed $m$ in this formula. We remark that the vectors on which we act with the braid action have the following form:
$$v_0\otimes v_{i_2} \otimes ... \otimes v_{i_{n}}, \text{ where } v_{i_2} \otimes ... \otimes v_{i_{n}}\in V^N_{n-1,m}. \ \ \ $$ Further on, the action of $K$ on this weight space is given just by a scalar:
$$K\curvearrowright V^N_{n-1,m}|_{\eta_{\xi_N,\lambda}}=\xi_{N}^{(n-1)\lambda-2m} \cdot Id.$$
Gluing back the weight zero vector, we obtain:
\begin{equation}
\left( Id \otimes K^{1-N}\right) \curvearrowright  v_0 \otimes V^N_{n-1,m}|_{\eta_{\xi_N,\lambda}}= \xi_{N}^{(n-1)(1-N)\lambda-2m(1-N)} \cdot Id
\end{equation}

We arrive at the following weighted sum:
\begin{equation}
 \Phi_{N}(L,\lambda)={\xi_N}^{(N-1)\lambda w(\beta_n)} \  \cdot \mathlarger{\sum}_{m=0}^{(N-1)(n-1)} \sum_{{\substack {(i_2,...,i_{n-1}) \\   \ \ \ \ \ \ \in E^N_{n-1,m}}}}
\end{equation} 
\begin{equation*}
 \left(\xi_{N}^{(n-1)(1-N)\lambda-2m(1-N)} C\left( v_{0}\otimes v_{i_2} \otimes ... \otimes v_{i_{n}},  \varphi^N_{n,m}|_{\eta_{\xi_N,\lambda}}(\beta_n)(v_{0}\otimes v_{i_2} \otimes ... \otimes v_{i_{n}} \right)\hspace{-1mm} \right).
\end{equation*}

After we separate the coefficients, we have:
 \begin{equation}
 \Phi_{N}(L,\lambda)={\xi_N}^{(N-1)\lambda w(\beta_n)} \ \xi_{N}^{(n-1)(1-N)\lambda}  \cdot \mathlarger{\sum}_{m=0}^{(N-1)(n-1)} \xi_{N}^{-2m(1-N)}  
\end{equation} 
\begin{equation*}
 \sum_{{\substack {(i_2,...,i_{n-1}) \\   \ \ \ \ \ \ \in E^N_{n-1,m}}}}  C\left( v_{0}\otimes v_{i_2} \otimes ... \otimes v_{i_{n}},  \varphi^N_{n,m}|_{\eta_{\xi_N,\lambda}}(\beta_n)(v_{0}\otimes v_{i_2} \otimes ... \otimes v_{i_{n}} \right).
\end{equation*}

Using the identification between specialised homological representations on the $N^{th}$ homological part and  specialised representations on the $N^{th}$ weight spaces from Corollary \ref{C:2}, we conclude the formula:
\begin{equation}\label{eq:6'}
 \Phi_{N}(L,\lambda)={\xi_N}^{(N-1)\lambda w(\beta_n)} \ \xi_{N}^{(n-1)(1-N)\lambda}  \mathlarger{\sum}_{m=0}^{(N-1)(n-1)} \xi_{N}^{-2m(1-N)}  
\end{equation} 
\begin{equation*}
\cdot  \sum_{{\substack {(i_2,...,i_{n-1}) \\   \ \ \ \ \ \ \in E^N_{n-1,m}}}}  C\left( \F_{0,i_1,...,i_{n}},  {}^{N} \! l_{n,m}(\beta_n) \F_{0,i_1,...,i_{n}} \right).
\end{equation*}

Now, looking at the second sum, we remark that it leads exactly to the homological partial trace introduced in the definition \ref{D:4}:
\begin{equation}\label{eq:7}
 \sum_{{\substack {(i_2,...,i_{n-1}) \\   \ \ \ \ \ \ \in E^N_{n-1,m}}}}  C\left( \F_{0,i_1,...,i_{n}},  {}^{N} \! l_{n,m}(\beta_n) \F_{0,i_1,...,i_{n}} \right) =hptr_0\left( {}^{N} \! l_{n,m} (\beta_n) \right)
\end{equation}
The last two equations \eqref{eq:6'} and \eqref{eq:7} lead exactly to the formula from Theorem \ref{THEOREM} and conclude the proof of the homological model.

\bibliography{biblio}{}
\bibliographystyle{plain}

\
\
\url{https://www.maths.ox.ac.uk/people/cristina.palmer-anghel}  

\end{document}